\documentclass{article}
\usepackage{amsmath,amsthm,amssymb,bbm}

\newcommand{\assign}{:=}
\newcommand{\tmop}[1]{\ensuremath{\operatorname{#1}}}
\newcommand{\tmstrong}[1]{\textbf{#1}}
\numberwithin{equation}{section}  
\newtheorem{corollary}{Corollary}
\newtheorem{definition}{Definition}
\newtheorem{proposition}{Proposition}
\newtheorem{theorem}{Theorem}
\newtheorem{lemma}{Lemma}

\newcommand{\XXint}[3]{{\setbox}0=\text{\ensuremath{#1 #2 #3 \int}}
{\vcenter{\text{\ensuremath{#2 #3}}}}{\kern}-.5{\tmwd}0}

\newcommand{\opn}[2]{\newcommand{\1}{\}} {\opn}{\Rm{Rm}} {\opn}{\Ric{Ric}}
{\opn}{\Rc{Rc}} {\opn}{\Scal{Sc}} {\opn}{\Tr{Tr}} {\opn}{\Trac{Tr}}
{\opn}detdet {\opn}{\diam{diam}} {\opn}{\dist{dist}} {\opn}{\Im}Im
{\opn}{\div}div {\opn}{\Ker{Ker}} {\opn}expexp {\opn}{\Vol{Vol}}
{\opn}{\exph{exph}} {\opn}{\Herm{Herm}} {\opn}{\End{End}} {\opn}{\Hess{Hess}}
{\opn}{\Vol{Vol}}}

\newcommand{\I}{\mathbb{I}}

\newcommand{\contract}{{\kern}-1.5pt{\vrule} width6.0pt height0.4pt depth0pt
{\vrule} width0.4pt height4.0pt depth0pt}
\newcommand{\retract}{{\kern}-1.5pt{\vrule} width0.4pt height4.0pt depth0pt
{\vrule} width6.0pt height0.4pt depth0pt}
\newcommand{\Openbox}{{\leavevmode} {\text{{\hfil}{\vrule}
width{\boxrulethickness} {\vbox} to{\Openboxwidth{{\advance}{\Openboxwidth}
-2{\boxrulethickness} {\hrule} height {\boxrulethickness}
width{\Openboxwidth}{\vfil} {\hrule} height{\boxrulethickness}}}{\vrule}
width{\boxrulethickness}{\hfil} }}}

\begin{document}

\title{The Soliton-K\"ahler-Ricci Flow over Fano Manifolds}\author{\\
{\tmstrong{NEFTON PALI}}}\maketitle

\begin{abstract}
  We introduce a flow of K\"ahler structures over Fano manifolds with
  formal limit at infinite time a K\"ahler-Ricci soliton. This flow correspond
  to a Perelman's modified backward K\"ahler-Ricci type flow that we call
  Soliton-K\"ahler-Ricci flow. It can be generated by the Soliton-Ricci flow.
  We assume that the Soliton-Ricci flow exists for all times and the
  Bakry-Emery-Ricci tensor preserve a positive uniform lower bound with respect
  to the evolving metric. In this case we show that the corresponding
  Soliton-K\"ahler-Ricci flow converges exponentially fast to a K\"ahler-Ricci
  soliton.
\end{abstract}

\section{Introduction}

This paper is the continuation of the work \cite{Pal2} in the K\"ahler setting. 

The notion of K\"ahler-Ricci soliton (in short KRS) is a natural generalization 
of the notion of K\"ahler-Einstein metric.
A KRS over a Fano manifold $X$ is a K\"ahler metric in the class $2 \pi c_1
(X)$ such that the gradient of the default potential of the metric to be
K\"ahler-Einstein is holomorphic. The terminology is justified by the fact
that the pull back of the KRS metric via the flow of automorphisms generated
by this gradient provides a K\"ahler-Ricci flow.

We remind that the K\"ahler-Ricci flow (in short KRF) has been introduced by
H. Cao in \cite{Cao}. In the Fano case it exists for all positive times. Its
convergence in the classic sense implies the existence of a K\"ahler-Einstein
metric. The fact that not all Fano manifolds admit K\"ahler-Einstein metrics
implies the non convergence in the classic sense of the KRF in general.

Our approach for the construction of K\"ahler-Ricci solitons is based on the study of
a flow of K\"ahler structures $(X, J_t, g_t)_{t \ge 0}$ associated to any
normalized smooth volume form $\Omega > 0$ that we will call
$\Omega$-Soliton-K\"ahler-Ricci flow (in short $\Omega$-$\tmop{SKRF}$). (See
the definition \ref{def-SKRF} below.) Its formal limit is precisely the KRS
equation with corresponding volume form $\Omega$.

It turns out that this flow is generated by the backward KRF via the
diffeomorphisms flow corresponding to the gradient of functions satisfying
Perelman's backward heat equation \cite{Per} for the KRF. In particular our point
of view gives a new reason for considering Perelman's backward heat equation.

Using a result in \cite{Pal1} we can show that the $\Omega$-SKRF can be generated
by the $\Omega$-Soliton-Ricci flow (in short $\Omega$-$\tmop{SRF}$) introduced
in \cite{Pal2} via an ODE flow of complex structures of Lax type. (See corollary 2 below.)

Let $\cal M$ be the space of smooth Riemannian metrics. We have explained in
\cite{Pal2} that it make sense to consider the $\Omega$-$\tmop{SRF}$ for a special
set $\mathcal{S}_{_{^{\Omega, +}}}^K\subset
\cal M$ of initial data (see \cite{Pal2} for the
definition) that we call {\tmstrong{positive scattering data}} with center of
polarization $K$. 

In this paper we denote by $\mathcal{K}_{_{^J}}$ the set of
$J$-invariant K\"ahler metrics. We define the set of positive K\"ahler
scattering data as the set 
$$
\mathcal{S}^{K, +}_{_{^{\Omega, J}}} \;\,\assign\;\,
\mathcal{S}_{_{^{\Omega, +}}}^K \cap \mathcal{K}_{_{^J}}\,.
$$
With this
notations hold the following result which is a consequence of the convergence
result for the $\Omega$-$\tmop{SRF}$ obtained in \cite{Pal2}.

\begin{theorem}
  \label{Main-Teo}Let $(X, J_0)$ be a Fano manifold and assume there exist
  $g_0 \in \mathcal{S}^{K, +}_{_{^{\Omega, J_0}}}$, for some smooth volume
  form $\Omega > 0$ and some center of polarization $K$, such that the
  solution $(g_t)_t$ of the $\Omega$-$\tmop{SRF}$ with initial data $g_0$
  exists for all times and satisfies $\tmop{Ric}_{g_t} (\Omega) \geqslant
  \delta g_t$ for some uniform bound $\delta \in \mathbbm{R}_{> 0}$. 

Then the
  corresponding solution $(J_t, g_t)_{t \geqslant 0}$ of the
  $\Omega$-$\tmop{SKRF}$ converges exponentially fast with all its space
  derivatives to a $J_{\infty}$-invariant K\"ahler-Ricci soliton $g_{\infty} =
  \tmop{Ric}_{g_{\infty}} (\Omega)$.
  
  Furthermore assume there exists a positive K\"{a}hler scattering data $g_0 \in \mathcal{S}^{K, +}_{_{^{\Omega, J_0}}}$
  with $g_0 J_0 \in 2 \pi c_1 (X)$ such that the evolving complex structure $J_t$ stays
  constant along a solution $(J_t, g_t)_{t \in [0, T)}$ of the
  $\Omega$-$\tmop{SKRF}$ with initial data $(J_0, g_0)$. Then $g_0$ is a
  $J_0$-invariant K\"ahler-Ricci soliton and $g_t \equiv g_0$.
\end{theorem}


\section{The Soliton-K\"ahler-Ricci Flow}

Let $\Omega > 0$ be a smooth volume form over an oriented Riemannian manifold
$(X, g)$. We remind that the $\Omega$-Bakry-Emery-Ricci tensor of $g$ is
defined by the formula
$$
\tmop{Ric}_g (\Omega)\;\; : =\;\; \tmop{Ric} (g)\;\, +\;\, \nabla_g\, d \log
   \frac{dV_g}{\Omega} \;. 
$$
A Riemannian metric $g$ is called a $\Omega$-Shrinking Ricci soliton (in short
$\Omega$-ShRS) if $g = \tmop{Ric}_g (\Omega)$. Let now $(X, J)$ be a complex
manifold. A $J$-invariant K\"ahler metric $g$ is called a $J$-K\"ahler-Ricci
soliton (in short $J$-KRS) if there exist a smooth volume form $\Omega > 0$
such that $g = \tmop{Ric}_g (\Omega)$.

The discussion below will show that if a compact K\"ahler manifold admit a
K\"ahler-Ricci soliton $g$ then this manifold is Fano and the choice of
$\Omega$ corresponding to $g$ is unique up to a normalizing constant.

We remind first that any smooth volume form $\Omega > 0$ over a complex
manifold $(X, J)$ of complex dimension $n$ induces a hermitian metric
$h_{\Omega}$ over the canonical bundle $K_{_{X, J}} \assign \Lambda_{_J}^{n,
0} T^{\ast}_{_X} $ given by the formula
$$
 h_{\Omega} (\alpha, \beta) \;\;\assign\;\; \frac{n! \,i^{n^2} \alpha \wedge
   \bar{\beta} }{\Omega} \;. 
$$
By abuse of notations we will denote by $\Omega^{- 1}$ the metric $h_{\Omega}$. The dual metric $h_{\Omega}^{\ast}$ on the anti-canonical bundle $K^{-
1}_{_{X, J}} = \Lambda_{_J}^{n, 0} T_{_X}$ is given by the formula
$$
 h_{\Omega}^{\ast} (\xi, \eta) \;\;=\;\; (- i)^{n^2} \Omega \left( \xi,
   \bar{\eta} \right) / n! \;. 
$$
Abusing notations again, we denote by $\Omega$ the dual metric
$h_{\Omega}^{\ast}$. We define the $\Omega$-Ricci form
$$
\tmop{Ric}_{_J} \left( \Omega \right) \;\;\assign\;\; i\,\mathcal{C}_{\Omega}  \big(
   K^{- 1}_{_{X, J}} \big) \;\;=\;\; -\;\, i\,\mathcal{C}_{\Omega^{- 1}} \big( K_{_{X,
   J}} \big)\;, 
$$
where $\mathcal{C}_h (L)$ denotes the Chern curvature of a hermitian line
bundle. In particular we observe the identity $\tmop{Ric}_{_J} (\omega) =
\tmop{Ric}_{_J} (\omega^n) .$ We remind also that for any $J$-invariant
K\"ahler metric $g$ the associated symplectic form $\omega \assign g J$
satisfies the elementary identity
\begin{eqnarray}
  \tmop{Ric} (g) \;\; = \;\; - \;\, \tmop{Ric}_{_J} (\omega) J
  \; .\label{cx-Ric-ide}
\end{eqnarray}
Moreover for all twice differentiable function $f$ hold the identity
\begin{eqnarray*}
  \nabla_g \,d\,f \;\; = \;\; - \;\, \big(i\, \partial_{_J} \bar{\partial}_{_J}
  \,f \big) J \;\,+\;\, g\, \bar{\partial}_{_{T_{X, J}}} \nabla_g \,f\; .
\end{eqnarray*}
(See the decomposition formula (\ref{cx-dec-Hess}) in the appendix.) We infer
the decomposition identity
\begin{equation}
  \label{cx-dec-Ric} \tmop{Ric}_g (\Omega) \;\;=\;\; -\;\, \tmop{Ric}_{_J} (\Omega) J \;\,+\;\, g\,
  \bar{\partial}_{_{T_{X, J}}} \nabla_g \log \frac{dV_g}{\Omega} \;.
\end{equation}
Thus a $J$-invariant K\"ahler metric $g$ is a $J$-KRS iff there exist a smooth
volume form $\Omega > 0$ such that
\begin{eqnarray*}
 \left\{ \begin{array}{l}
\displaystyle{g \;\,=\;\, - \tmop{Ric}_{_J} (\Omega) J }\;,\\
     \\
\displaystyle{\bar{\partial}_{_{T_{X, J}}} \nabla_g \log \frac{\hspace{0.25em} d
     V_g}{\Omega} \;\,=\;\, 0 }\; .
   \end{array} \right. 
\end{eqnarray*}
The first equation of this system implies that $(X, J)$ must be a Fano
variety. We can translate the notion of K\"ahler-Ricci soliton in symplectic
therms. In fact let $(X, J_0)$ be a Fano manifold of complex dimension $n$,
let $c_1 : = c_1 (X, [J_0])$, where $[J_0]$ is the co-boundary class of the complex structure $J_0$
and set
$$
 \mathcal{J}^+_{X, J_0} \;\;: =\;\; \left\{ J \in [J_0] \hspace{0.25em} \mid
   \hspace{0.25em} N_J = 0\,,\; \exists\, \omega\, \in\, \mathcal{K}^{2 \pi c_1}_{_{^J}}
   \right\}\;, 
$$
where $N_J$ denotes the Nijenhuis tensor and
$$
\mathcal{K}^{2 \pi c_1}_{_{^J}} \;\;: =\;\; \Big\{ \omega \in 2 \pi c_1 \mid
   \omega \;=\; J^{\ast} \omega J\,,\;\, -\;\omega J \;>\; 0 \Big\}\;, 
$$
is the set of $J$-invariant K\"ahler forms $\omega \in 2 \pi c_1$. It is clear
that for any complex structure $J \in \mathcal{J}^+_{X, J_0}$ and any form
$\omega \in \mathcal{K}^{2 \pi c_1}_{_{^J}}$ there exist a unique smooth
volume form $\Omega > 0$ with $\int_X \Omega = (2 \pi c_1)^n$ such that $\omega
= \tmop{Ric}_{_J} (\Omega)$.

This induces an inverse functional $\tmop{Ric}_{_J}^{- 1}$ such that $\Omega
= \tmop{Ric}_{_J}^{- 1} (\omega)$. With this notation we infer that a
$J$-invariant form $\omega \in 2 \pi c_1$ is the symplectic form associated to a $J$-KRS if and only if $0 < g :
= - \omega J$ and
$$
 \bar{\partial}_{_{T_{X, J}}} \nabla_g \log
   \frac{\omega^n}{\tmop{Ric}_{_J}^{- 1} (\omega)} \;\;= \;\;0\;. 
$$
In equivalent volume therms we say that a smooth volume form $\Omega > 0$ with
$\int_X \Omega = (2 \pi c_1)^n$ is a $J$-Soliton-Volume-Form (in short
$J$-SVF) if
\begin{eqnarray*}
 \left\{ \begin{array}{l}
\displaystyle{0 \;\,<\;\, g \;\,: =\;\, - \;\tmop{Ric}_{_J} (\Omega) J}\;,\\
     \\
\displaystyle{\bar{\partial}_{_{T_{X, J}}} \nabla_g \log \frac{\tmop{Ric}_{_J}
     (\Omega)^n}{\Omega} \;\,=\;\, 0 }\;.
   \end{array} \right. 
\end{eqnarray*}
We deduce a natural bijection between the sets $\{g \mid J \text{- \tmop{KRS}}
\}$ and $\{\Omega \mid J \text{- \tmop{SVF}} \}$. We define also the set of
Soliton-Volume-Forms over $(X,J_0)$ as
$$
 \mathcal{SV}_{X, J_0} \;\;: =\;\; \left\{ \Omega \;>\; 0 \mid \int_X \Omega \;=\; (2 \pi
   c_1)^n\;,\;\, \exists \, J \in \mathcal{J}^+_{X, J_0} \;:\; \Omega\;
   \text{\tmop{is} a} \;J \text{- SVF} \right\} \;.
$$
We would like to investigate under which conditions $\mathcal{SV}_{X, J_0} 
\neq \emptyset$. For this purpose it seem natural to consider the
following flow of K\"ahler structures.

\begin{definition}
\label{def-SKRF}
{\tmstrong{$($The $\Omega$-Soliton-K\"ahler-Ricci flow$)$}}.
Let $(X, J_0)$ be
a Fano manifold and let $\Omega > 0$ be a smooth volume form with $\int_X
\Omega = (2 \pi c_1)^n$. A $\Omega$-Soliton-K\"ahler-Ricci flow $($in short
$\Omega$-$\tmop{SKRF})$ is a flow of K\"ahler structures $(X, J_t,
\omega_t)_{t \geqslant 0}$ which is solution of the evolution system
\begin{equation}
  \label{VSKRF-eq}  \left\{ \begin{array}{l}
\displaystyle{    \frac{d}{dt} \hspace{0.25em} \omega_t\; \;=\;\; \tmop{Ric}_{_{J_t}} (\Omega) \;\,-\;\,
    \omega_t }\;,\\
    \\
\displaystyle{    \frac{d}{dt} \hspace{0.25em} J_t \;\;= \;\; J_t \hspace{0.25em}
    \bar{\partial}_{_{T_{X, J_t}}} \nabla_{g_t} \log
    \frac{\omega_t^n}{\Omega}}\;,
  \end{array} \right.
\end{equation}
where $g_t : = - \,\omega_t J_t$.
\end{definition}
The $\Omega$-SKRF equation/system (\ref{VSKRF-eq}) can be written in an equivalent way as
\begin{equation}
  \label{SKRF-eq}  \left\{ \begin{array}{l}
\displaystyle{\frac{d}{dt} \hspace{0.25em} \omega_t \;\,-\;\, i\, \partial_{_{J_t}}
    \bar{\partial}_{_{J_t}} \,f_t\;\; =\;\; \tmop{Ric}_{_{J_t}} (\omega_t) \;-\;
    \omega_t  }\;,\\
    \\
\displaystyle{    \frac{d}{dt}\hspace{0.25em} J_t \;\;=\;\; J_t\, \bar{\partial}_{_{T_{X, J_t}}} \nabla_{g_t}
    f_t}\;,\\
    \\
 \displaystyle{   e^{- f_t} \omega^n_t \;\;=\;\; \Omega }\;.
  \end{array} \right.
\end{equation}
We observe also that (\ref{VSKRF-eq}) or (\ref{SKRF-eq}) are equivalent to the system
\begin{equation}
  \label{VSKRF-ev}  \left\{ \begin{array}{l}
\displaystyle{\frac{d}{dt}\; \omega_t \;\;=\;\; \tmop{Ric}_{_{J_t}} (\Omega) \;\,-\;\, \omega_t}\;,\\
    \\
\displaystyle{    J_t \;\;: =\;\; (\Phi_t^{- 1})^{\ast} J_0 \;\;: = \;\;\Big[ \left( d \Phi_t \cdot J_0
    \right) \circ \Phi_t^{- 1} \Big] \cdot d \Phi_t^{- 1}}\;,\\
    \\
\displaystyle{\frac{d}{dt} \;\Phi_t \;\;=\;\; -\;\, \left( \frac{1}{2}\, \nabla_{g_t} \log
    \frac{\omega_t^n}{\Omega} \right) \circ \Phi_t}\;,\\
    \\
    \Phi_0 \;\;=\;\; \text{Id}_X\; .
  \end{array} \right.
\end{equation}
In fact lemma \ref{t-derLIE} combined with lemma \ref{Lie-CXst} in the appendix implies
\begin{eqnarray*}
  \frac{d}{dt} \; (\Phi_t^{\ast} J_t) & = & \Phi_t^{\ast} \left(
  \frac{d}{dt}\; J_t \;\,-\;\, \frac{1}{2}\, L_{\nabla_{g_t} f_t} J_t \right)\\
  &  & \\
  & = & \Phi_t^{\ast} \left( \frac{d}{dt} \; J_t \;\,-\;\, J_t
  \hspace{0.25em} \bar{\partial}_{_{T_{X, J_t}}} \nabla_{g_t} f_t \right)
  \;\,=\;\, 0\;.
\end{eqnarray*}
We define now $\hat{\omega}_t : = \Phi_t^{\ast} \omega_t$ , $\hat{g}_t : =
\Phi_t^{\ast} g_t = - \,\hat{\omega}_t  J_0$ and we observe that the evolving family
$$
(J_0 , \hat{\omega}_t)_t \;\; =\;\;\Phi_t^{\ast} (J_t , \omega_t)_t \;,
$$
represents a backward K\"ahler-Ricci flow over $X$. In fact the K\"ahler condition
$$
\nabla_{\hat{g}_t} J_0 \;\;=\;\;\Phi_t^{\ast}
   \big(\nabla_{g_t} J_t\big) \;\;=\;\; 0\;, 
$$
hold and
\begin{eqnarray*}
  \frac{d}{dt} \hspace{0.25em} \hat{\omega}_t & = & \Phi_t^{\ast} \left(
  \frac{d}{dt} \hspace{0.25em} \omega_t \;\,-\;\, \frac{1}{2}  \hspace{0.25em}
  L_{\nabla_{g_t} f_t } \hspace{0.25em} \omega_t \right)\\
  &  & \\
  & = & \Phi_t^{\ast} \Big( \tmop{Ric}_{_{J_t}} (\omega_t) \;\,-\;\, \omega_t
  \Big)\\
  &  & \\
  & = & \tmop{Ric}_{_{J_0}} ( \hat{\omega}_t) \hspace{0.75em} -
  \hspace{0.75em} \hat{\omega}_t \hspace{0.25em},
\end{eqnarray*}
by the formula (\ref{cx-Lie}) in the appendix. We observe that the volume form preserving
condition $e^{- f_t} \omega^n_t = \Omega$ in the equation (\ref{SKRF-eq}) is
equivalent to the heat equation
\begin{equation}
  \label{SKRF-gag-pot-EV} 2\, \frac{d}{dt} \hspace{0.25em} f_t \;\;=\;\;
  \tmop{Tr}_{\omega_t}  \frac{d}{dt} \hspace{0.25em} \omega_t
  \;\; =\;\; -\;\, \Delta_{g_t} f_t \;\,+\;\,  \tmop{Scal} (g_t)
  \;\,-\;\,  2 \,n\;,
\end{equation}
with initial data $f_0 : = \log \frac{\omega_0^n}{\Omega}$. 
(In this paper we adopt the sign
convention $\Delta_g \assign - \tmop{div}_g \nabla_g$.)
In its turn this
is equivalent to the heat equation
\begin{equation}
  \label{gaug-pot-EVOL} 2\, \frac{d}{dt} \hspace{0.25em}  \hat{f}_t
  \;\;=\;\; -\;\, \Delta_{\hat{g}_t}  \hat{f}_t
  \;\,-\;\, | \nabla_{\hat{g}_t}  \hat{f}_t
  |_{\hat{g}_t}^2 \;\,+\;\, \tmop{Scal} ( \hat{g}_t)
  \;\,-\;\,  2 \,n\;,
\end{equation}
with same initial data $\hat{f}_0 : = \log \frac{\omega_0^n}{\Omega}$. In fact
let $\hat{f}_t : = f_t \circ \Phi_t$ and observe that the evolution equation
of $\Phi_t$ in (\ref{VSKRF-ev}) implies
\begin{eqnarray*}
  \frac{d}{dt} \hspace{0.25em}  \hat{f}_t & = & \left( \frac{d}{dt}
  \hspace{0.25em} f_t \right) \circ \Phi_t \;\,+ \;\,
  \left\langle (\nabla_{g_t} f_t) \circ \Phi_t \,,
   \frac{d}{dt} \hspace{0.25em} \Phi_t \right\rangle_{g_t \circ
  \Phi_t}\\
  &  & \\
  & = & \left( \frac{d}{dt} \hspace{0.25em} f_t \;\,-\;\,
   \frac{1}{2}  \hspace{0.25em} | \nabla_{g_t} f_t |^2_{g_t}
  \right) \circ \Phi_t \; .
\end{eqnarray*}
We observe also that the derivation identity
$$
 0 \;\; = \;\; \frac{d}{dt}  \left( \Phi_t^{- 1} \circ
   \Phi_t \right) 
\;\;=\;\;  \left( \frac{d}{dt} \hspace{0.25em}
   \Phi_t^{- 1} \right) \circ \Phi_t 
\;\,+ \;\, d\,
   \Phi_t^{- 1} \cdot \frac{d}{dt} \hspace{0.25em} \Phi_t\;, 
$$
combined with the evolution equation of $\Phi_t$ in the system (\ref{VSKRF-ev}) implies
$$
2 \hspace{0.25em} d\, \Phi_t \cdot \left( \frac{d}{dt}  \hspace{0.25em}
   \Phi_t^{- 1} \right) \circ \Phi_t \;\; = \;\;
   \nabla_{g_t} f_t \circ \Phi_t \;\; = \;\; d\, \Phi_t
   \cdot \nabla_{\hat{g}_t}  \hat{f}_t \;. 
$$
We infer the evolution formula
\begin{equation}
  \label{diff-flw}  \frac{d}{dt} \hspace{0.25em} \Phi_t^{- 1} \;\;
  = \;\; \frac{1}{2} \left( \nabla_{\hat{g}_t}  \hat{f}_t \right)
  \circ \Phi_t^{- 1} \;.
\end{equation}
In conclusion we deduce that the $\Omega$-SKRF $(J_t \hspace{0.25em},
\hspace{0.25em} \omega_t)_{t \geqslant 0}$ is equivalent to the system of independent equations
\begin{eqnarray*}
 \left\{ \begin{array}{l}
\displaystyle{     \frac{d}{dt} \;\hat{\omega}_t \;\;=\;\; \tmop{Ric}_{_{J_0}} (
     \hat{\omega}_t) \;\,-\;\, \hat{\omega}_t }\;,\\
     \\
\displaystyle{      2 \, \frac{d}{dt}\;  \hat{f}_t \;\;=\;\; -\;\, \Delta_{\hat{g}_t} 
     \hat{f}_t \;\,-\;\, | \nabla_{\hat{g}_t}  \hat{f}_t |_{\hat{g}_t}^2 \;\,+\;\, \tmop{Scal}
     ( \hat{g}_t) \;\,-\;\, 2\, n}\;, \\
     \\
     e^{- \hat{f}_0}  \hat{\omega}^n_0 \;\;=\;\; \Omega\;,
   \end{array} \right. 
\end{eqnarray*}
by means of the gradient flow of diffeomorphisms (\ref{diff-flw}).
\\
\\
{\tmstrong{Notation.}} Let $(X, g, J)$ be a K\"ahler manifold with symplectic
form $\omega \assign g J$ and consider $v \in S_{_{\mathbbm{R}}}^2
T^{\ast}_X,$ $\alpha \in \Lambda_{_{\mathbbm{R}}}^2 T^{\ast}_X .$ We define
the endomorphisms $v^{\ast}_g \assign g^{- 1} v$ and $\alpha_g^{\ast} \assign
\omega^{- 1} \alpha .$ For example we will define the endomorphisms
$$
\tmop{Ric}^{\ast}_g (\Omega) \;\;\assign\;\; g^{- 1} \tmop{Ric}_g (\Omega)\;,
$$ 
and
$$
\tmop{Ric}^{\ast}_{_J} (\Omega)_g \;\;: =\;\; \omega^{- 1} \tmop{Ric}_{_J} (\Omega)\;.
$$
With this notations formula (\ref{cx-dec-Ric}) implies the decomposition
identity
\begin{equation}
  \label{dec-end-Ric} \tmop{Ric}^{\ast}_{g_{_{}}} (\Omega) \;\;=\;\;
  \tmop{Ric}^{\ast}_{_J} (\Omega)_g \;\,+\;\, \bar{\partial}_{_{T_{X, J}}}
  \nabla_g \log \frac{d V_g}{\Omega}\; .
\end{equation}

\section{The Riemannian nature of the Soliton-K\"ahler-Ricci Flow}{}

The goal of this section is to show that the K\"ahler structure along the SKRF
comes for free from the SRF introduced in \cite{Pal2} by means of a Lax type ODE
for the complex structures which preserves the K\"ahler condition. For this
purpose let $(J_t, g_t)_{t \geqslant 0}$ be a $\Omega$-SKRF. Time deriving
the identity $g_t = - \,\omega_t J_t$ we obtain
\begin{eqnarray*}
  \frac{d}{dt} \hspace{0.25em} g_t & = & - \;\,\frac{d}{dt} 
  \hspace{0.25em} \omega_t \,J_t \;\,- \;\, \omega_t \,
  \frac{d}{dt} \hspace{0.25em} J_t\\
  &  & \\
  & = & - \;\, \tmop{Ric}_{_{J_t}} (\Omega) J_t \;\, +\;\,
  \omega_t\, J_t \;\, - \;\, \omega_t \,J_t
  \, \bar{\partial}_{_{T_{X, J_t}}} \nabla_{g_t} f_t\\
  &  & \\
  & = & - \;\, \tmop{Ric}_{_{J_t}} (\Omega) J_t \;\,+\;\,
 g_t \, \bar{\partial}_{_{T_{X, J_t}}} \nabla_{g_t} f_t
  \;\,- \;\, g_t\\
  &  & \\
  & = & \tmop{Ric}_{g_t} (\Omega) \;\, - \;\, g_t\;,
\end{eqnarray*}
thanks to the complex decomposition (\ref{cx-dec-Ric}). We have obtained the
evolving system of K\"ahler structures $(J_t, g_t)_{t \geqslant 0}$,
\begin{equation}
  \label{Rm-SKRF-eq}  \left\{ \begin{array}{l}
\displaystyle{    \frac{d}{dt}  \hspace{0.25em} g_t \;\; = \;\;
    \tmop{Ric}_{g_t} (\Omega) \;\, - \;\, g_t}
    \;,\\
    \\
\displaystyle{   2 \hspace{0.25em} \frac{d}{dt} \hspace{0.25em} J_t \;\; =\;\;
    \hspace{0.75em} \hspace{0.25em} J_t \nabla^2_{g_t} \log
    \frac{dV_{g_t}}{\Omega} \;\,-\;\,\nabla^2_{g_t}
    \log \frac{dV_{g_t}}{\Omega} \hspace{0.25em} J_t }\;,
  \end{array} \right.
\end{equation}
which is equivalent to (\ref{VSKRF-eq}). (The second equation in the system follows from the
fact that in the K\"ahler case the Chern connection coincides with the
Levi-Civita connection.) We observe that the identity (\ref{cx-Ric-ide}) implies that the Ricci endomorphism
$$
\tmop{Ric}^{\ast}(g) \;\;=\;\;
  \tmop{Ric}^{\ast}_{_J} (\omega)_g\;,
$$
is $J$-linear.
Thus the system (\ref{Rm-SKRF-eq}) is equivalent to
the evolution of the couple $(J_t, g_t)$ under the system
\begin{equation}
  \label{Rm-SKRF-ode}  \left\{ \begin{array}{l}
\displaystyle{    \dot{g}_t \;\;=\;\; \tmop{Ric}_{g_t} (\Omega)
    \;\,-\;\, g_t }\;,\\
    \\
\displaystyle{    2 \hspace{0.25em} \dot{J}_t \;\;=\;\;
     J_t \, \dot{g}_t^{\ast} \;\,-\;\,
    \dot{g}_t^{\ast} J_t }\;,\\
    \\
\displaystyle{    J^2_t \;\;=\;\; - \;\,\I_{T_X} \;,
    \hspace{1em} (J_t)^T_{g_t} \;\;=\;\; - \;\,J_t
    \;, \hspace{1em} \nabla_{g_t} J_t \;\;=\;\;
     0 }\;,
  \end{array} \right.
\end{equation}
where for notation simplicity we set $\dot{g}_t \assign \frac{d}{dt} g_t$ and
$\dot{J}_t \assign \frac{d}{dt} J_t$. Moreover $(J_t)^T_{g_t}$ denotes the transpose of $J_t$ with respect to $g_t$. We remind now an elementary fact (see
lemma 4 in \cite{Pal1}).

\begin{lemma}
  Let $(g_t)_{t \geqslant 0}$ be a smooth family of Riemannian metrics and let
  $(J_t)_{t \geqslant 0}$ be a family of endomorphisms of $T_X$ solution of
  the $\tmop{ODE}$
$$
2\, \dot{J}_t \;\;=\;\; J_t\,  \dot{g}_t^{\ast} \;\,-\;\,\dot{g}_t^{\ast} J_t\;, 
$$
  with initial conditions $J^2_0 \;=\; -\; \I_{T_X}$
  and $(J_0)^T_{g_0} \;=\; -\; J_0$. Then this
  conditions are preserved in time i.e. $J^2_t \;=\;-\; \I_{T_X}$ and $(J_t)^T_{g_t} \;=\;-\; J_t$ for all $t \geqslant 0$. 
\end{lemma}

We deduce that the system (\ref{Rm-SKRF-ode}) is equivalent to the system
\begin{equation}
  \label{Rm-SKRF-ode-Kah}  \left\{ \begin{array}{l}
    \dot{g}_t \;\;=\;\; \tmop{Ric}_{g_t} (\Omega)
    \;\,-\;\, g_t \;,\\
    \\
    2 \hspace{0.25em} \dot{J}_t \;\;=\;\;
    \hspace{0.25em} J_t \, \dot{g}_t^{\ast} \;\,-\;\,
     \dot{g}_t^{\ast}  J_t \;,\\
    \\
    \nabla_{g_t} J_t \;\;=\;\; 0 \;,
  \end{array} \right.
\end{equation}
with K\"ahler initial data $(J_0, g_0) .$ We show now how we can get rid of
the last equation. We define the vector space
\begin{eqnarray*}
  \mathbbm{F}_g \;\;\assign \;\; \Big\{ v \in C^{\infty} \left( X,
  S_{_{\mathbbm{R}}}^2 T^{\ast}_X \right) \mid \hspace{0.25em} \nabla_{_{T_X,
  g}} v_g^{\ast} \;\,=\;\, 0 \Big\}\;,
\end{eqnarray*}
where $\nabla_{_{T_X, g}}$ denotes the covariant exterior derivative acting on
$T_X$-valued differential forms and we remind the following key result
obtained in \cite{Pal1}.

\begin{proposition}
  \label{F-Kah}Let $(g_t)_{t \geqslant 0}$ be a smooth family of Riemannian
  metrics such that $\dot{g}_t \in \mathbbm{F}_{g_t}$ and let $(J_t)_{t
  \geqslant 0}$ be a family of endomorphisms of $T_X$ solution of the
  $\tmop{ODE}$
$$
\dot{J}_t \;\;=\;\; J_t  \,\dot{g}_t^{\ast} \;\,-\;\,
     \dot{g}_t^{\ast} J_t\;, 
$$
  with K\"ahler initial data $(J_0, g_0)$. Then $(J_t, g_t)_{t \geqslant 0}$
  is a smooth family of K\"ahler structures.
\end{proposition}

In particular using lemma 1 in \cite{Pal2} we infer the following corollary which
provides a simple way to generate K\"ahler structures.

\begin{corollary}
  Let $(J_t, g_t)_{t \geqslant 0} \subset C^{\infty} (X,
  \tmop{End}_{_{\mathbbm{R}}} (T_X)) \times \mathcal{M}$ be the solution of
  the ODE
\begin{equation}
\left\{ \begin{array}{l}
\displaystyle{\frac{d}{d t} \hspace{0.25em} \dot{g}_t^{\ast} \;\;=\;\; 0}
       \;,\\
       \\
\displaystyle{2 \hspace{0.25em} \frac{d}{d t}\hspace{0.25em} J_t \;\;=\;\;
       \hspace{0.25em} J_t\,  \dot{g}_t^{\ast} \;\,-\;\,
       \dot{g}_t^{\ast} J_t }\;,
     \end{array} \right. 
\end{equation}
  with $(J_0, g_0)$ K\"ahler data and with $\nabla_{_{T_X, g_0}} (
  \dot{g}_0^{\ast})^p = 0$ for all $p \in \mathbbm{Z}_{> 0}$. Then $(J_t,
  g_t)_{t \geqslant 0}$ is a smooth family of K\"ahler structures.
\end{corollary}

We remind also the definitions introduced in \cite{Pal2}. We define the set of
pre-scattering data
\begin{eqnarray*}
  \mathcal{S}_{_{^{\Omega}}} \;\; := \;\; \Big\{ g \in \mathcal{M} \mid \hspace{0.25em}
  \nabla_{_{T_X, g}} \tmop{Ric}^{\ast}_g (\Omega) \;\;=\;\; 0 \Big\}\; .
\end{eqnarray*}
\begin{definition}
  {\tmstrong{$($The $\Omega$-Soliton-Ricci flow$)$}}. Let $\Omega > 0$ be a
  smooth volume form over an oriented Riemannian manifold $X$. A
  $\Omega$-Soliton-Ricci Flow $($in short $\Omega$-$\tmop{SRF})$ is a Flow of
  Riemannian metrics $(g_t)_{t \geqslant 0} \subset
  \mathcal{S}_{_{^{\Omega}}}$ solution of the evolution equation $ \dot{g}_t =
  \tmop{Ric}_{g_t} (\Omega) - g_t$.
\end{definition}

From Proposition \ref{F-Kah} we deduce the following fact which shows the
Riemannian nature of the $\Omega$-SKRF. Namely that the $\Omega$-SKRF can be generated
by the $\Omega$-SRF.

\begin{corollary}
  \label{Split-SKRF}Let $\Omega > 0$ be a smooth volume form over a K\"ahler
  manifold $(X, J_0)$ and let $(g_t)_{t \geqslant 0}$ be a solution of the
  $\Omega$-$\tmop{SRF}$ with K\"ahler initial data $(J_0, g_0) .$ Then the
  family $(J_t, g_t)_{t \geqslant 0}$ with $(J_t)_{t \geqslant 0}$ the
  solution of the $\tmop{ODE}$
$$
2\, \dot{J}_t \;\;=\;\; J_t \, \dot{g}_t^{\ast} \;\,-\;\,
     \dot{g}_t^{\ast} J_t\;, 
$$
  is a solution of the $\Omega$-$\tmop{SKRF}$ equation.
\end{corollary}

\section{The set of K\"ahler pre-scattering data}

We define the set of K\"ahler pre-scattering data as
$\mathcal{S}_{_{^{\Omega, J}}} \assign \mathcal{S}_{_{^{\Omega}}} \cap
\mathcal{K}_{_{^J}}$. Using the complex decomposition formula
(\ref{dec-end-Ric}) we infer the equality
\begin{eqnarray*}
  \mathcal{S}_{_{^{\Omega, J}}} \;\; = \;\; \Big\{ g \in \mathcal{K}_{_{^J}} \mid
  \hspace{0.25em} \partial^g_{_{T_{X, J}}} \bar{\partial}_{_{T_{X, J}}}
  \nabla_g \log \frac{d V_g}{\Omega} \;\;=\;\; -\;\, \bar{\partial}_{_{T_{X, J}}}
  \tmop{Ric}^{\ast}_{_J} (\Omega)_g \Big\}\; .
\end{eqnarray*}
In fact the identity $d \tmop{Ric}_{_J} (\Omega) = 0$ is
equivalent to the identity $\partial_{_J} \tmop{Ric}_{_J} (\Omega) = 0,$ which
in its turn is equivalent to the identity
\begin{eqnarray*}
  \partial_{_{T_{X, J}}}^g \tmop{Ric}^{\ast}_{_J} (\Omega)_g \;\; = \;\; 0\; .
\end{eqnarray*}
We observe now the following quite elementary facts.
\begin{lemma}
  Let $(X, J)$ be a Fano manifold and let $g \in \mathcal{S}_{_{^{\Omega,
  J}}}$ such that $\omega \assign g J \in 2 \pi c_1 (X)$. Then $g$ is a
  $J$-invariant KRS iff $\omega = \tmop{Ric}_{_J} (\Omega)$, iff $\tmop{Ric}_g
  (\Omega)$ is $J$-invariant.
\end{lemma}

\begin{proof}
  In the case $\tmop{Ric}^{\ast}_{_J} (\Omega)_g =\mathbbm{I}_{T_X}$ the condition
  $g \in \mathcal{S}_{_{^{\Omega, J}}}$ is equivalent to the condition
  \begin{equation}
    \label{dbar-clos}  \bar{\partial}_{_{T_{X, J}}} \nabla_g \log \frac{d
    V_g}{\Omega} \;\;=\;\; 0\;.
  \end{equation}
  (i.e the $J$-invariance of $\tmop{Ric}_g (\Omega)$ and thus that $g$ is a
  $J$-invariant KRS.) In fact in this case
  \begin{eqnarray*}
    \partial^g_{_{T_{X, J}}} \bar{\partial}_{_{T_{X, J}}} \nabla_g \log
    \frac{d V_g}{\Omega} \;\; = \;\; 0\;,
  \end{eqnarray*}
  which by a standard K\"ahler identity implies
  \begin{eqnarray*}
    \bar{\partial}^{\ast_g}_{_{T_{X, J}}} \bar{\partial}_{_{T_{X,
    J}}} \nabla_g \log \frac{d V_g}{\Omega} \;\;= \;\; 0\; .
  \end{eqnarray*}
  Thus an integration by parts yields the required identity (\ref{dbar-clos}).
  
  On the other hand if we assume that $\tmop{Ric}_g (\Omega)$ is
  $J$-invariant i.e we assume (\ref{dbar-clos}) then the condition $g \in
  \mathcal{S}_{_{^{\Omega, J}}}$ is equivalent to the condition
  \begin{eqnarray*}
    \bar{\partial}_{_{T_{X, J}}} \tmop{Ric}^{\ast}_{_J} (\Omega)_g \;\; = \;\;
    0 \;.
  \end{eqnarray*}
  For cohomology reasons hold the identity
  \begin{eqnarray*}
    \omega \;\; = \;\; \tmop{Ric}_{_J} (\Omega) \;\,+\;\, i\, \partial_{_J}
    \bar{\partial}_{_J} u\;,
  \end{eqnarray*}
  for some $u \in C^{\infty} (X, \mathbbm{R})$. We deduce the equalities
  \begin{eqnarray*}
    0 \;\; = \;\; \bar{\partial}_{_{T_{X, J}}}  \left( i \partial_{_J}
    \bar{\partial}_{_J} u \right)_g^{\ast} \;\;=\;\; \bar{\partial}_{_{T_{X,
    J}}} \partial^g_{_{T_{X, J}}} \nabla_g \,u \;.
  \end{eqnarray*}
  Using again a standard K\"ahler identity we infer
  \begin{eqnarray*}
    \partial^{\ast_g}_{_{T_{X, J}}} \partial^g_{_{T_{X, J}}} \nabla_g \,u \;\;=\;\;
    0 \;.
  \end{eqnarray*}
An integration by parts yields the conclusion $i \,\partial_{_J}
  \bar{\partial}_{_J} u \equiv 0$, i.e $u \equiv 0$, which implies the
  required KRS equation.
\end{proof}

\begin{lemma}
  \label{baby-SKRF}Let $(X, J)$ be a K\"ahler manifold and let $(g_t)_{t \in
  [0, T)}$ be a smooth family of $J$-invariant K\"ahler metrics solution of
  the equation
  \begin{eqnarray*}
    \dot{g}_t \;\; = \;\; \tmop{Ric}_{g_t} (\Omega) \;\,-\;\, g_t\; .
  \end{eqnarray*}
Then this family is given by the formula 
$$
g_t \;\;=\;\; -\;\, \tmop{Ric}_{_J} (\Omega) J
  \;\,+\;\, \big(g_0 \;\,+\;\, \tmop{Ric}_{_J} (\Omega) J\big)\, e^{- t}\,,
$$
with $J$-invariant K\"ahler
  initial data $g_0$ solution of the equation
  \begin{eqnarray*}
    \bar{\partial}_{_{T_{X, J}}} \nabla_{g_0} \log \frac{d
    V_{g_0}}{\Omega} \;\; = \;\; 0\; .
  \end{eqnarray*}
\end{lemma}

\begin{proof}
  The fact that $g_t$ is $J$-invariant implies that $\dot{g}_t$ is also
  $J$-invariant. Then the decomposition formula (\ref{cx-dec-Ric}) combined
  with the evolution equation of $g_t$ provides
  \begin{eqnarray*}
    \dot{g}_t \;\; = \;\; -\;\, \tmop{Ric}_{_J} (\Omega) J\;\, -\;\, g_t\;,
  \end{eqnarray*}
which implies the required conclusion.
\end{proof}
From the previous lemmas we deduce directly the following corollary.
\begin{corollary}
  Let $g_0 \in \mathcal{S}_{_{^{\Omega, J}}}$ with $g J \in 2 \pi c_1 (X)$ be
  an initial data for the $\Omega$-SKRF such that the complex structure stays
  constant along the flow. Then $g_0$ is a $J$-invariant KRS and $g_t \equiv
  g_0 = - \tmop{Ric}_{_J} (\Omega) J$.
\end{corollary}

The last statement in the theorem \ref{Main-Teo} follows directly from this
corollary.

\section{On the smooth convergence of the Soliton-K\"ahler-Ricci flow}

We show now the convergence statement in theorem 1.
According to the convergence result for the $\Omega$-SRF obtained in \cite{Pal2} we
just need to show the smooth convergence of the complex structures. We
consider the differential system
\begin{eqnarray*}
  \label{cv-Cx}  \left\{ \begin{array}{l}
    2\, \dot{J}_t \;\;=\;\; \left[ J_t\,, \dot{g}_t^{\ast} \right]\;,\\
    \\
    2\, \ddot{g}_t \;\;=\;\; -\;\, \Delta^{^{_{_{\Omega}}}}_{g_t}  \dot{g}_t \;\,-\;\, 2\, \dot{g}_t\;,
  \end{array} \right.
\end{eqnarray*}
along the $\Omega$-SRF (see \cite{Pal2}) and we remind the uniform estimates 
\begin{eqnarray*}
| \dot{g}_t
|_{g_t} &\leqslant& | \dot{g}_0 |_{C^0 (X), g_0} \,e^{- \,\delta\, t / 2}\;,
\\
\\
e^{- C} g_0 &\leqslant& g_t \;\;\leqslant\;\; e^C g_0\;,
\end{eqnarray*}
proved in \cite{Pal2}. We consider also the
estimate of the norm
\begin{eqnarray*}
  | \dot{J}_t |_{g_t} \;\; \leqslant \;\; \sqrt{2\, n}\, |J_t |_{g_t} \,| \dot{g}^{\ast}_t
  |_{g_t}\;,
\end{eqnarray*}
where the constant $\sqrt{2 \,n}$ comes from the equivalence between the
Riemannian norm and the operator norm on the space of endomorphisms of $T_{X,
x}$. We observe now the trivial identities
\begin{eqnarray*}
  |J_t |^2_{g_t} \;\; = \;\; \tmop{Tr}_{_{\mathbbm{R}}} \Big[ J_t \,(J_t)_{g_t}^T
  \Big] \;\;=\;\; -\;\, \tmop{Tr}_{_{\mathbbm{R}}} J^2_t \;\,=\;\, \tmop{Tr}_{_{\mathbbm{R}}}
  \mathbbm{I}_{T_X} \;\;=\;\; 2\, n\; .
\end{eqnarray*}
We deduce the exponential estimate of the variation of the complex structure
\begin{eqnarray*}
  | \dot{J}_t |_{g_t} \;\; \leqslant \;\; 2\, n\,| \dot{g}_0 |_{C^0 (X), g_0} \,e^{-\,
  \delta\, t / 2}\;,
\end{eqnarray*}
and thus the convergence of the integral
\begin{eqnarray*}
  \int^{+ \infty}_0 | \dot{J}_t |_{g_0} d t \;\; < \;\; +\;\, \infty\; .
\end{eqnarray*}
In its turn this shows the existence of the integral
\begin{eqnarray*}
  J_{\infty} \;\; \assign \;\; J_0 \;\,+\;\, \int^{+ \infty}_0 \dot{J}_t \,d t\;,
\end{eqnarray*}
thanks to Bochner's theorem. Moreover hold the exponential estimate
\begin{eqnarray*}
  |J_{\infty} \;-\; J_t |_{g_0} \;\; \leqslant \;\; \int^{+ \infty}_t | \dot{J}_s
  |_{g_0} \,d s \;\;\leqslant\;\; C' \,e^{- \,\delta\, t / 2}\; .
\end{eqnarray*}
On the other hand the K\"ahler identity $\nabla_{g_t} J_t \equiv 0$ implies the equality
$$
2\,
\nabla^p_{g_t}  \dot{J}_t \;\;=\;\; \left[ J_t, \nabla^p_{g_t}  \dot{g}_t^{\ast}
\right]\;,
$$ for all $p \in \mathbbm{N}$. We deduce the estimates
\begin{eqnarray*}
  | \nabla^p_{g_t}  \dot{J}_t |_{g_t} \;\; \leqslant \;\; \sqrt{2\, n}\, |J_t |_{g_t} \,|
  \nabla^p_{g_t}  \dot{g}^{\ast}_t |_{g_t} \;\;\leqslant\;\; 2\, n\, C_p \,e^{-\,
  \varepsilon_p t}\,,
\end{eqnarray*}
thanks to the exponential decay of the evolving Riemannian metrics proved in \cite{Pal2}. The fact that the flow of Riemannian metrics $(g_t)_{t \geqslant 0}$ is
uniformly bounded in time for any $C^p (X)$-norm implies the uniform estimate
\begin{eqnarray*}
  | \nabla^p_{g_0}  \dot{J}_t |_{g_0} \;\;\leqslant \;\; C'_p\, e^{-\, \varepsilon_p t}\,
  .
\end{eqnarray*}
We infer the convergence of the integral
\begin{eqnarray*}
  \int^{+ \infty}_0 | \nabla^p_{g_0}  \dot{J}_t |_{g_0} d t \;\; < \;\; +\;\, \infty\;,
\end{eqnarray*}
and thus the existence of the integral
\begin{eqnarray*}
  I_p \;\; \assign \;\; \nabla^p_{g_0} J_0 \;\,+\;\, \int^{+ \infty}_0 \nabla^p_{g_0} 
  \dot{J}_t \,d t \;.
\end{eqnarray*}
We deduce the exponential estimate
\begin{eqnarray*}
  |I_p \;-\; \nabla^p_{g_0} J_t |_{g_0} \;\; \leqslant \;\; \int^{+ \infty}_t |
  \nabla^p_{g_0}  \dot{J}_s |_{g_0} d s \;\;\leqslant\;\; C_p' \,e^{-\, \varepsilon_p t} \;.
\end{eqnarray*}
A basic calculus fact combined with an induction on $p$ implies $I_p =
\nabla^p_{g_0} J_{\infty}$. We deduce that $(J_{\infty}, g_{\infty})$ is a
K\"ahler structure. Then the convergence result in \cite{Pal2} implies that
$g_{\infty}$ is a $J_{\infty}$-invariant KRS.

\section{Appendix. Basic differential identities}
The results explained in this appendix are well known. We include them here for readers convenience.
\begin{lemma}
  \label{t-derLIE}Let $M$ be a differentiable manifold and let
$$
(\xi_t)_{t \geqslant 0} \;\;\subset \;\;C^{\infty} (M, T_M)\,,\quad (\alpha_t)_{t
     \geqslant 0}\;\; \subset\;\; C^{\infty} \left( M, (T^{\ast}_M)^{\otimes p}
     \otimes T_M^{\otimes r} \right)\,, 
$$
be smooth families and let $(\Phi_t)_{t \geqslant 0}$ be the flow of
  diffeomorphisms induced by the family $(\xi_t)_{t \geqslant 0}$ , i.e
$$
 \frac{d}{dt} \hspace{0.25em} \Phi_t \;\;=\;\; \xi_t \circ \Phi_t\,,\quad \Phi_0 \;\;=\;\;
     \tmop{Id}_M \;. 
$$
  Then hold the derivation formula
$$
\frac{d}{dt} \hspace{0.25em} \big(\Phi_t^{\ast} \alpha_t\big)\;\; =\;\; \Phi_t^{\ast} 
     \left( \frac{d}{dt} \hspace{0.25em} \alpha_t \;\,+\;\, L_{\xi_t} \alpha_t \right)\;,
$$
\end{lemma}
\begin{proof}
  We prove first the particular case
  \begin{equation}
    \label{arb-flw-LIE}  \frac{d}{dt} \vphantom{dt}_{|t = 0} \hspace{0.25em}
    \big(\Phi_t^{\ast} \alpha\big)\;\; =\;\; L_{\xi_0} \alpha \hspace{0.25em},
  \end{equation}
  where $\alpha$ is $t$-independent. For this purpose we consider the
  $1$-parameter subgroup of diffeomorphisms $(\Psi_t)_{t \geqslant 0}$ induced by
  $\xi_0$, i.e
$$
 \frac{d}{dt} \hspace{0.25em} \Psi_t \;\;=\;\;
     \xi_0 \circ \Psi_t \;,\quad \Psi_0 \;\;=\;\; \tmop{Id}_M \; . 
$$
  Let $\hat{\Psi} : \mathbbm{R}_{\geqslant 0} \times M \longrightarrow M$ given by
  $\hat{\Psi} (t, x) = \Psi^{- 1}_t (x)$ and observe the equalities
  \begin{equation}
    \label{inv-flw}  \frac{d}{dt} \vphantom{dt}_{|t = 0} \hspace{0.25em}
    \Psi_t^{- 1} \;\;=\;\;- \;\, \xi_0
    \;\;=\;\; \frac{d}{dt} \vphantom{dt}_{|t = 0}
    \hspace{0.25em} \Phi_t^{- 1}\,,
  \end{equation}
  We will note by $\partial$ the partial derivatives of the coefficients of the
  tensors with respect to a trivialization of the tangent bundle over an open
  set $U \subset M$. Let $v \in T^{\otimes p}_{M, x}$. Then
  \begin{eqnarray*}
    (L_{\xi_0} \alpha) \cdot v & = & \left( \frac{d}{dt} \vphantom{dt}_{|t =
    0} \hspace{0.25em} \Psi_t^{\ast} \alpha \right) \cdot v\\
    &  & \\
    & = & \frac{d}{dt} \vphantom{dt}_{|t = 0} \hspace{0.25em} \Big[ (d
    \Psi_t^{- 1})^{\otimes r} \cdot (\alpha \circ \Psi_t) \Big] \cdot (d
    \Psi_t)^{\otimes p} \cdot v\\
    &  & \\
    & = & \left[ \frac{d}{dt} \vphantom{dt}_{|t = 0} \hspace{0.25em}
    (\partial_x \hat{\Psi})^{\otimes r} (t, \Psi_t (x)) \right] \cdot \alpha
    \cdot v\\
    &  & \\
    & + & \frac{d}{dt} \vphantom{dt}_{|t = 0} \hspace{0.25em} \Big[ (\alpha
    \circ \Psi_t) \cdot (d \Psi_t)^{\otimes p} \cdot v \Big]\\
    &  & \\
    & = & (\partial_t \,\partial_x \hat{\Psi})^{\otimes r} (0, x) \cdot \alpha
    \cdot v \;\,+\;\, \left( (\partial_x^2 \,
    \hat{\Psi})^{\otimes r} (0, x) \cdot \xi^{\otimes r}_0 (x) \right) \cdot
    \alpha \cdot v\\
    &  & \\
    & + & \left( \frac{d}{dt} \vphantom{dt}_{|t = 0} \hspace{0.25em} \alpha
    (\Psi_t (x)) \right) \cdot v \;\,+\;\, \alpha (x)
    \cdot (\partial_t\, \partial_x \Psi)^{\otimes p} (0, x) \cdot v\\
    &  & \\
    & = & - \;\,(\partial_x \,\xi_0)^{\otimes r} (x) \cdot \alpha
    \cdot v\\
    &  & \\
    & + & \left( \partial_x \,\alpha \,(x) \cdot v \right) \cdot \xi_0 (x)
    \;\,+\;\, \alpha (x) \cdot (\partial_x\,
    \xi_0)^{\otimes p} (x) \cdot v \hspace{0.25em},
  \end{eqnarray*}
  since the map
  \[ (\partial_x^2  \,\hat{\Psi})^{\otimes r} (0, x) : S^2 T_U^{\otimes r}
     \longrightarrow T_U^{\otimes r} \hspace{0.25em}, \]
  is zero. Observe in fact the identity
$$
\partial_x \Psi^{- 1}_0 \;\;=\;\; \tmop{Id}_{T_U}\;.
$$
Moreover the same computation and conclusion work for $\Phi_t$ thanks to
  (\ref{inv-flw}). We infer the identity (\ref{arb-flw-LIE}). We prove now the
  general case. We expand the time derivative
  \begin{eqnarray*}
    \frac{d}{dt} \hspace{0.25em} \big(\Phi_t^{\ast} \alpha_t\big) & = & \frac{d}{ds}
    \vphantom{ds}_{|s = 0} \hspace{0.25em} \Phi_{t + s}^{\ast} \,\alpha_{t +
    s}\\
    &  & \\
    & = & \Phi_t^{\ast} \left( \frac{d}{dt} \hspace{0.25em} \alpha_t \right)
    \;\,+\;\,\frac{d}{ds} \vphantom{ds}_{|s = 0}
    \hspace{0.25em} \Phi_{t + s}^{\ast} \,\alpha_t\\
    &  & \\
    & = & \Phi_t^{\ast} \left( \frac{d}{dt} \hspace{0.25em} \alpha_t \right)
    \;\,+\;\, \frac{d}{ds} \vphantom{ds}_{|s = 0}
    \hspace{0.25em} (\Phi_t^{- 1} \Phi_{t + s})^{\ast} \,\Phi_t^{\ast} \,\alpha_t
    \hspace{0.25em} .
  \end{eqnarray*}
  We set $\Phi^t_s : = \Phi_t^{- 1} \Phi_{t + s}$ and we observe the equalities
\begin{eqnarray*}
\frac{d}{ds} \vphantom{ds}_{|s = 0} \hspace{0.25em} \Phi^t_s
&=&
 d\, \Phi_t^{- 1} \cdot \frac{d}{ds}
     \vphantom{ds}_{|s = 0} \hspace{0.25em} \Phi_{t + s} 
\\
\\
&=&
d\, \Phi_t^{- 1} \cdot (\xi_t \circ \Phi_t) 
\\
\\
&=& \Phi_t^{\ast}\, \xi_t \; . 
\end{eqnarray*}
  Then the identity (\ref{arb-flw-LIE}) applied to the family $(\Phi^t_s)_s$
  implies
\begin{eqnarray*}
\frac{d}{dt} \hspace{0.25em} \big(\Phi_t^{\ast} \alpha_t\big) 
&=&
\Phi_t^{\ast} \left( \frac{d}{dt} \hspace{0.25em}
     \alpha_t \right) \;\,+\;\, L_{\Phi_t^{\ast}
     \xi_t}  \Phi_t^{\ast} \,\alpha_t 
\\
\\
&=& \Phi_t^{\ast} \left( \frac{d}{dt} \hspace{0.25em}
     \alpha_t \;\,+\;\,L_{\xi_t} \alpha_t \right)\;. 
\end{eqnarray*}
\end{proof}

\begin{lemma}
  \label{Lie-CXst}Let $(X, J)$ be an almost complex manifold and let $N_{_J}$
  be the Nijenhhuis tensor. Then for any $\xi \in C^{\infty} (X, T_X)$ hold
  the identity
  \begin{equation}
    \label{Lie-CX-st} L_{\xi} J \;\;=\;\; 2\, J \left( \bar{\partial}_{_{T_{X, J}}}
    \xi \;\,-\;\, \xi\; \neg\; N_{_J} \right) \;.
  \end{equation}
\end{lemma}

\begin{proof}
  Let $\eta \in C^{\infty} (X, T_X)$. Then
  \begin{eqnarray*}
    \hspace{0.25em} \bar{\partial}_{_{T_{X, J}}} \xi \,(\eta) & = &
    [\eta^{0, 1}, \xi^{1, 0}]^{1, 0} \;\,+\;\,
    [\eta^{1, 0}, \xi^{0, 1}]^{0, 1} \;,\\
    &  & \\
    N_{_J} (\xi, \eta) & = & [\xi^{1, 0}, \eta^{1, 0}]^{0, 1} \;\,+\;\, [\xi^{0, 1}, \eta^{0, 1}]^{1, 0} \;.
  \end{eqnarray*}
and the conclusion follows by decomposing in type $(1, 0)$ and $(0, 1)$ the
  identity
  \[ (L_{\xi} J) \,\eta \;\;=\;\; [\xi, J \eta] \;\,-\;\, J [\xi, \eta]\; . \]
\end{proof}

We observe now that if $(X, J, \omega)$ is a K\"ahler manifold and $u \in C^{\infty} (X,
\mathbbm{R})$, then hold the identities
\[ \nabla_{\omega} u \;\neg\; \omega \;\;=\;\; -\;\, (d\, u) \cdot J \;\;=\;\; -\; i\, \partial_{_J} u \;\,+\;\, i\,
   \bar{\partial}_{_J} u\;, \]
and
\begin{equation}
  \label{cx-Lie} L_{\nabla_{\omega} u} \,\omega \;\;=\;\; d \left( \nabla_{\omega} \,u
\;\neg\; \omega \right) \;\,=\;\, 2\, i\, \partial_{_J} \bar{\partial}_{_J} u \;.
\end{equation}
\begin{lemma}
  Let $(X, J, g)$ be a K\"ahler manifold and let $u \in C^{\infty} (X,
  \mathbbm{R})$. Then hold the decomposition formula
  \begin{equation}
    \label{cx-dec-Hess} \nabla_g\, d\,u \;\;=\;\; i\, \partial_{_J} \bar{\partial}_{_J}
    u \left( \cdot, J \cdot \right) \;\,+\;\, g \left( \cdot,
    \bar{\partial}_{_{T_{X, J}}} \nabla_g \,u \,\cdot \right) \;.
  \end{equation}
\end{lemma}

\begin{proof}
  Let $\xi \hspace{0.25em}, \eta \hspace{0.25em}, \mu \in C^{\infty} (X,
  T_X)$. By definition of Lie derivative hold the identity
  \[ \xi \,.\, g (\eta, \mu) \;\;=\;\;
 (L_{\xi}\, g) (\eta, \mu) \;\,+\;\,
     g (L_{\xi} \,\eta, \mu) \;\,+\;\, g (\eta, L_{\xi}\,
     \mu) \; . \]
  Let $\omega : = g (J \cdot, \cdot)$ be the induced K\"ahler form. Then by
  using again the definition of Lie derivative we infer the equalities
  \begin{eqnarray*}
\xi \,.\, g (\eta, \mu) & = & \xi
    \,.\, \omega (\eta, J \mu)\\
    &  & \\
    & = & (L_{\xi} \,\omega) (\eta, J \mu) \;\,+\;\,
    \omega (L_{\xi} \,\eta, J \mu) \;\,+\;\, \omega
    (\eta, L_{\xi} (J \mu))\\
    &  & \\
    & = & (L_{\xi} \,\omega) (\eta, J \mu) \;\,+\;\, g
    (L_{\xi} \,\eta, \mu) \hspace{0.75em} + \hspace{0.75em} \omega (\eta,
    (L_{\xi} \,J) \mu) \;\,+\;\, g (\eta, L_{\xi} \,\mu)\; .
  \end{eqnarray*}
  We deduce the identity
$$
 L_{\xi} \,g \;\;=\;\; L_{\xi} \,\omega \,(\cdot, J
     \cdot) \;\,+\;\, \omega \,(\cdot, L_{\xi}\, J \cdot)
     \; . 
$$
  We apply this identity to the vector field $\xi : = \nabla_g \,u$. Then the conclusion follows
  from the identity
$$
 L_{\nabla_g u} \hspace{0.25em} g \;\;=\;\; 2\, \nabla_g\, d\,u \hspace{0.25em}, 
$$
  combined with (\ref{cx-Lie}), and (\ref{Lie-CX-st}). 
\end{proof}

\vspace{1cm}
\noindent
Nefton Pali
\\
Universit\'{e} Paris Sud, D\'epartement de Math\'ematiques 
\\
B\^{a}timent 425 F91405 Orsay, France
\\
E-mail: \textit{nefton.pali@math.u-psud.fr}

\end{document}